\documentclass[sn-mathphys]{sn-jnl}

\usepackage[utf8]{inputenc}

\usepackage{array , subfig }
\usepackage{graphicx}
\usepackage[english]{babel}
\usepackage{comment}
\usepackage{amsmath} 
\title{Bibliography management: BibTeX}
\usepackage [ utf8 ]{inputenc }
\usepackage{amsmath, amsthm, amscd, amsfonts, mathrsfs, amssymb, graphicx, color, diagmac2}
\usepackage[all]{xy}
\usepackage{xcolor}

\usepackage{soul}
\usepackage{ulem}
\usepackage{setspace}

\def\r{\hbox{${ \mathbb R}$}}

\newcommand{\di}{\displaystyle}
\newcommand{\mumax}{\mu_{\max}}
\newcommand{\s}{\mathcal{S}}

\jyear{2023}%

\theoremstyle{thmstyleone}%
\newtheorem{theorem}{Theorem}
\newtheorem{proposition}[theorem]{Proposition}%

\theoremstyle{thmstyletwo}%

\theoremstyle{thmstylethree}%

\begin{document}

\title[ A Complete Mathematical Model For  Trichoderma Fungi Kinetics]{ A Complete Mathematical Model For  Trichoderma Fungi Kinetics}


\author*[1]{\fnm{Asmae} \sur{Hardoul}}\email{asmae.hardoul@uit.ac.ma}

\author[2]{\fnm{Zoubida} \sur{Mghazli}}\email{zoubida.mghazli@uit.ac.ma}

\affil[1]{\orgdiv{Ibn Tofail University, Faculty of Sciences, \'Equipe d'Ing\'enierie MAth\'ematique (E.I.MA.)}, \orgname{  PDE, Algebra and Spectral Geometry Laboratory}, \orgaddress{
\city{Kenitra}, \postcode{BP 133}, 
\country{Morocco}}}

\abstract{We develop an unstructured mathematical model to describe the growth kinetics of the Trichoderma fungus and its enzyme (cellulase) production in the rhizosphere. This model incorporates hydrolysis, where organic matter is broken down, producing a liquefied substrate that supports fungal growth and enzyme synthesis. The resulting 
Captures the interactions between substrate degradation, fungal growth, and enzyme formation. Analysis of the model's asymptotic behavior reveals convergence to a global attractor comprising infinite non-hyperbolic equilibria, depending on initial conditions.  Numerical simulations with data from the literature confirm the theoretical study and validate the
model.}




\maketitle
\section{Introduction}\label{sec1}
Trichoderma is a fungus that grows in almost all soils and is a constituent of the fungal communities of the rhizosphere. They are also present on the root surfaces of many plants \cite{Harman 1998}.
 They play an important role in the soil nutrient cycle. Indeed,  they 
 have a beneficial effect on plant growth and soil
 fertility,  
 \cite{2007, 2016}. 
  Fungi of the genus Trichoderma have been used as biological control agents against a wide spectrum of soil-borne pathogens.
Their antagonistic behavior has been demonstrated in numerous researches.   
\cite{davet, 42005, 2005, 2009}.
    It has also been shown that these fungi play a key role in the production of a wide range of hydrolytic extracellular enzymes such as cellulase, which allows the degradation of cellulose  
    \cite{Schuster}. These enzymes are also involved in the suppression of plant diseases \cite {Elad 1982} and generally have an antifungal effect. They are also highly synergistic in their antifungal activity when combined with fungicides whose mode of action affects the cell membranes of pathogenic fungi \cite{Lorito 1994}.
    
Kinetic models for the production of cellulase by Trichoderma reesei have been developed in various research \cite{ lo 2010,  Rakshit 1991, bader 1993,  Velkovska 1997, Muthuvelayudham 2007}, in which a soluble substrate has been used. For example, in the reference \cite{Rakshit 1991}, a kinetic model of cell growth and cellulase formation has been developed, without using a substrate consumption equation.
Muthuvelayudham and Viruthagiri \cite{Muthuvelayudham 2007}, found that the logistic model,  the Luedeking-Piret model and the integrated logistic model with Leudeking-Piret could accurately describe the cellulase production process, but they did not provide any numerical simulation. 
Bader et al \cite{bader 1993}, developed detailed kinetics for the production of an individual enzyme component, but cellulose data were not available although the model was provided.
 The most comprehensive kinetic models, including cell growth, substrate consumption, and cellulase production from cellulose, were first introduced by Velkovska et al \cite{Velkovska 1997}.
 Their study developed a simple model of product formation using total enzyme and variable substrate reaction models. However, the model was unable to predict the biomass curves. Furthermore, the effect of the substrate on cellulose production is not considered, and the rate of hydrolysis is assumed to be equal to the rate of consumption of the substrate throughout the process, which is usually not the case.
 Lijuan Ma et al \cite {ma 2013}, developed a simple  and complete model for the cellulase production process using cellulose, but did not consider the product equation. Experimental data on cellulose, biomass and cellulase were used for parameter estimation and validation of the model by comparing the simulations with the experimental results of (Velkovska et al. \cite{Velkovska 1997}).

Hydrolysis is a key step in the biodegradation of organic matter, since it converts complex compounds into soluble substrates that can be directly assimilated by microorganisms. As shown by Vavilin et al.~\cite{Vavilin}, in mechanistic models of solid matter degradation, this step is generally limiting for substrate availability. In the case of Trichoderma, which cannot directly utilize native cellulose, enzymatic hydrolysis by extracellular cellulases completely determines fungal growth. This is confirmed by recent studies on Trichoderma ~\cite{Duarte2021}, as well as several works showing that the substrate flux resulting from hydrolysis governs biomass dynamics and enzyme production (Bischof et al. \cite{bischof}; Li et al.\cite{li}; Qian et al.\cite{qian}). A hydrolysis constant that is too low,  therefore, limits substrate availability and slows microbial growth. Thus, it is essential to establish appropriate assumptions regarding this constant to ensure a realistic and stable representation of the process.

To the best of our knowledge, none of the mathematical models describing the growth kinetics of biomass (Trichoderma) and enzyme production (cellulase) during substrate degradation has taken the hydrolysis step into account.
 
 The objective of this paper is the development and analysis of a complete system that models the kinetics of enzyme production (cellulase with cellulose) by the  
 fungus Trichoderma in the rhizosphere. Inspired by the model of Lijuan Ma et al. \cite{ma 2013}, we develop an unstructured mathematical model that integrates the step of hydrolysis of organic matter and takes into account the formation of a product (cellulase), which makes it a more realistic and complete model.
A fraction of dead microorganisms can constitute a new substrate. According to \cite{alpha}, they are then recycled as organic matter to be hydrolyzed. In
 the proposed model, we consider the mortality parameter, as well as the substrate consumption maintenance parameter, and the product formation maintenance parameter.
  We analyze the asymptotic behavior of the dynamics of the system obtained and prove that we have a continuum of non-hyperbolic equilibria. An analogous analysis was performed in \cite{saleh},  for a problem related to an anaerobic digestion model in the landfill. We show that each trajectory of the system is bounded and converges towards one of these non-hyperbolic equilibria according to the initial conditions. Some numerical tests are presented.
 
 The article structure is as follows. The model and assumptions are introduced and discussed in the next section.
Section 2 gives a mathematical analysis of the asymptotic behaviour and the set of equilibria.
Finally, in Section 3, numerical simulations for different values of data from the literature and different initial conditions are given, confirming the theoretical study.  A discussion, conclusions, and bibliography conclude this article.
\section{ Model and hypotheses}
%
Modeling the growth and product formation of various micro-organisms is generally a very difficult task due to the complexity of living systems.
The number of factors that influence the ability of the organism to grow and produce enzymes is very large and leads to complex complex systems with many parameters.
In order to build a model that is accessible to mathematical analysis and numerical simulations, while taking into account the essential factors to describe the kinetics of enzyme production by the filamentous fungus Trichoderma in the rhizosphere, we will make the following assumptions.
As a first assumption, we assume that the overall biomass has a single morphological form and that it is fed by a single substrate. We also assume that the extracellular medium (soil) is perfectly homogeneous.  The model   is based on the principle of mass balance applied to   variables that are the concentrations  of the components of the system. We note by $X$, $B$, $s$ and $P$ the concentrations of organic matter,
 living biomass, 
substrate and
 product (enzyme), respectively.


The dynamics of the biomass $B$, taking into account the biological process of mortality,  are given by 

\begin{equation} \label{dB/dt = mu B -kd B}
\frac{dB}{dt} =  \mu(s)\, B- k_{d}\, B,
 \end{equation}
where $\mu(s) $ is the growth function and  $k_{d}$ is the specific mortality rate. 

 The concentration of the substrate $s$ and the cell growth are obviously interdependent. As the cells grow, they use up the substrate, which is thus depleted. The rate of substrate consumption is described by the equation :
 \begin{equation}\label{ s }
 \frac{ds}{dt}   
 = -  \frac{1}{Y_{B/s}} \mu(s) B,
\end{equation}
where $Y_{B/s}$ is the coefficient of efficiency of conversion of the substrate $s$ into the biomass $B$.
 The substrate utilization kinetics, given in Eq. \eqref{ s }, can be modified to take into account the cell maintenance, which means living consumes substrate:    
 \begin{equation}\label{ds/dt}
    \frac{ds}{dt}	= -  \>\frac{1}{Y_{B/s}} \>  \mu(s) \; B   -  m_{s}\> B,   
\end{equation}
Where $m_s$ is the coefficient of cell-specific maintenance by the substrate $s$.

To model the kinetics of the formation of the product $P$, we use the Luedeking-Piret model ( see \cite{Luedeking 1959}):
 \begin{align}\label{dp/dt}
    \frac{dP}{dt}	=  \frac{1}{Y_{P/s}} \>\mu(s) \; B  +  m_{P}\> B, 
\end{align}
where $Y_{P/s}$ is the conversion efficiency of a substrate $s$ to a product $P$, expressed as {\it"the quantity of $P$ formed per quantity of $s$ consumed"}, and $m_P$ is the cell-specific maintenance coefficient for the product $P$. We note that Equation \eqref{dp/dt} describing the formation of the product $P$, is related to the growth rate of $B$ and its maintenance.\\

We assume that the parameters satisfy the following assumptions.
\begin{description}
\item[\bf Hypothesis 1]  The specific growth rate function $\mu(\cdot)$ is of class $C^1$ with $\mu(0) = 0$ and $\mu(s) >0$ for all $s> 0$.
\item[\bf Hypothesis 2]   The coefficients $ k_d, Y_{B/s}, Y_{P/s}, m_s$ and $m_P $ satisfy the following conditions.
 \begin{enumerate}
     \item 
     $0< k_d<\displaystyle\max_s\mu(s)$

\item 
$0<Y_{B/s}<1\;\;$ and $\;\;0<Y_{P/s}<1$ 

 \item  
 $m_s>0\;\;$  and  $\;\;m_P>0$.
\end{enumerate}
\end{description}  All the conditions of {\bf Hypothesis 2} are verified in real biological phenomena.  
  A large number of growth rate functions satisfy {\bf Hypothesis 1}. The Monod law \cite{monod 1942} is the most widely used, and many adaptations have been made to this growth model, such as the Teissier law, the Contois law and the Dabes et al law \cite{law2}, where the growth rate is limited only by the substrate.   Other models take into account an inhibition of,  at least, one of the variables, such as the Haldane law, the Luong law, the Aiba law, the Han and Levenspiel law, the Yano law and the Luong law \cite{law1}.
 It is known that Trichoderma fungi have no inhibitors for their growth.   So, Monod's law is suitable for our problem. It takes the form:
\begin{equation}\label{monod}
     \mu(s)=\mumax\frac{s}{k_{s}+s},
\end{equation}
where $\mumax$ is the maximum growth rate and $k_s$ is the half-saturation constant.\\

Hydrolysis constitutes the key step in the valorization of lignocellulosic substrates because Trichoderma is not capable of directly assimilating native cellulose. Access to assimilable carbon depends exclusively on the extracellular enzymatic conversion of the polymeric matrix into soluble sugars (glucose, cellobiose), which subsequently feed intracellular metabolic pathways. Consequently, the efficiency of hydrolysis directly determines the amount of soluble substrate available to the biological system.
When the cellulolytic activity is high, the hydrolytic flux ensures a continuous supply of carbon, enabling sustained biomass growth. In contrast, slow hydrolysis leads to a limitation in soluble substrate, even when the physiological potential for fungal growth remains high. In this case, the limiting factor does not stem from an intrinsic metabolic cap but from an insufficient upstream hydrolytic flux.
This analysis is fully consistent with the experimental evidence available in the literature. Several studies have shown that the intensity of extracellular cellulase secretion by Trichoderma is strongly correlated with the solubilization of cellulose and the rate of biomass accumulation. Bischof et al. (2013)\cite{bischof} demonstrated that enhanced enzyme production results in a faster release of soluble sugars and a significant increase in cellular productivity. Similar findings were reported by Li et al. (2016)\cite{li} and Qian et al. (2017)\cite{qian}, who emphasized that improved hydrolytic performance systematically increases bioavailable carbon flux.
More recently, Duarte et al. (2021)\cite{Duarte2021} experimentally confirmed, using Trichoderma longibrachiatum isolates, that the efficiency of lignocellulosic forage hydrolysis directly controls the availability of soluble sugars as well as the growth and metabolic activity dynamics. Their results show that increasing cellulolytic activity immediately accelerates hydrolysis and enhances carbon release, thus validating the central role of the hydrolysis rate in the governance of the system.
Together, these observations, both classical and recent, demonstrate that hydrolysis acts as the primary control step of mycelial growth dynamics in mechanistic degradation models. In standard models, hydrolysis is described by first-order kinetics: 
\begin{equation}\label{kh}
    \frac{dX}{dt} =  -K_{H} X,
\end{equation}
where $K_H$ is the hydrolysis constant.
 According to 
  \cite{alpha}, a fraction $\alpha$ of dead microorganisms limited by biological constraints: $0\leq \alpha<1$ can constitute a new substrate and is then recycled as a material to be hydrolyzed. In this case, the last equation becomes :
\begin{equation}\label{eq1}
    \frac{dX}{dt} =  -K_{H} \> X +\alpha \> k_{d}\>B .
\end{equation}

Moreover, since hydrolysis transforms organic matter into a new substrate, the term $K_{H} X $ is therefore a source term in Eq. \eqref{ds/dt} which becomes
\begin{equation}\label{eqs}
    \frac{ds}{dt}	= -\left(\frac{1}{Y_{B/s}}  \mu(s)  + m_{s}\right) B  + K_{H}  X  .
\end{equation}



From equation \eqref{eqs}, the dynamics of the substrate is based on the opposition between the consumption induced by biomass, represented by $\left(\frac{1}{Y_{B/s}}  \mu(s)  + m_{s}\right) B$, and the flux of the substrate produced by hydrolysis, given by $K_{H}  X $.\\

To represent a biologically realistic situation in which Trichoderma has a carbon flux that supports its growth, hydrolysis production must be sufficient at all times to compensate for metabolic consumption. This requirement leads to the introduction of a hydrolysis efficiency, which states that the substrate flux resulting from hydrolysis must be sufficient to compensate for the consumption induced by the biomass.\\

{\bf Hypothesis 3 :} for all $t>0$:
\begin{align}\label{h3}
   K_H X(t) \;\geq\; \max_{s \in [0, s_{\max}]} \left\{ \frac{1}{Y_{B/s}} \,\mu(s) + m_s \right\} \, B(t) 
\end{align}
This condition reflects the biological reality of the system. For Trichoderma, growth can occur only if enough sugars are produced by hydrolysis of cellulose. If consumption by the biomass exceeds hydrolytic production, the available substrate becomes insufficient and growth slows down, not because of a physiological limitation of the fungus, but because of a lack of assimilable substrate. The introduced inequality therefore ensures that the model remains consistent with reality: it guaranties that the soluble substrate never becomes limiting, thus allowing growth that is consistent with experimental observations. {\bf Hypothesis 3} establishes a direct link between hydrolysis kinetics and microbial dynamics, thereby constituting an essential component for the accurate modeling of microbial growth and substrate utilization.\\

Building upon the equations that describe the underlying dynamics and the key assumptions outlined above, we construct a comprehensive mathematical model. This model comprises four coupled equations representing the temporal evolution of organic matter $X$, biomass $B$, substrate $s$, and product $P$:
   \begin{align} 
\left\{ \begin{array}{rcll}\label{model s1}
\displaystyle
\frac{dX}{dt} &=&  -K_{H}\, X + \alpha\, k_{d}\, B .\\
\displaystyle
   \frac{dB}{dt} &=&  \biggl(\mu(s) - k_{d}\biggr)\, B .\\
   \displaystyle
  \frac{ds}{dt}	&=& -\left(\di\frac{1}{Y_{B/s}}\, \mu(s)   + m_{s}\right)\,  B  + K_{H}\,X  . \\
  \displaystyle
    \frac{dP}{dt}	&=& \left(\di\frac{1}{Y_{P/s}}\, \mu(s)  + m_{P} \right)\, B,     
     \end{array}\right.
    \end{align}
     which can be written as
\begin{equation}\label{eqsNonHomogene}
    \dfrac{dV}{dt}=A(V)V,
\end{equation}
where
\begin{align}\label{AetX}
V=\begin{pmatrix}
     X\\B\\s\\P
     \end{pmatrix}\quad \mbox{ and }\quad
    A(V) = \begin{pmatrix}
-K_H & \alpha  k_{d} & 0 & 0  \\ 0 & \mu(s)  - k_{d} & 0 & 0 \\ K_H & -\left(\frac{1}{Y_{B/s}} \mu(s)+ m_{s}\right)& 0 & 0 \\
0 & \frac{1}{Y_{P/s}} \mu(s)+ m_{P}& 0 & 0
\end{pmatrix}.
 \end{align}

\section{Equilibrium points and    asymptotic behaviour}
 This section presents an analysis of the system \eqref{model s1}.  We show that the solution remains positive,   determine the equilibrium points, and give the asymptotic behavior of the solution.
\subsection{Positivity of the solution and equilibrium points}
We begin by proving the positivity of the solution.
\begin{proposition} \label{p0}
For any vector $(X_0,B_0,s_0,P_0)$   with non-negative components, the solution of the system \eqref{model s1} with the initial condition $(X(0),B(0),s(0),P(0))=(X_0,B_0,s_0,P_0)$ exists and is unique and non-negative.
\end{proposition}
%
%

{\begin{proof}
The existence and uniqueness of the solution to system \eqref{model s1} are guaranteed by the Cauchy-Lipschitz theorem, since all functions involved in this system are globally Lipschitzian (see \cite{walter}). {assuming $\mu(\cdot)$ is bounded and $C^1$, Hypothesis 1}

  To show the positivity of the solution of system \eqref{model s1}, let $X_0\geq 0$, $, B_0\geq 0$, $s_0\geq 0$, and $P_0\geq 0$ be given. 
We proceed by step
 and show,  first, that $B(t) \geq 0$, then we deduce that $X(t) \geq 0$,  and then finally  the positive result for $s(t)$ and $P(t)$.
\begin{enumerate}
\item[\*1/]
{The explicit solution of $B(t)$ is
\[
B(t) = B(0) \, \exp \left( \int_0^t \big( \mu(s(\tau)) - k_d \big) \, d\tau \right).
\]
Since $B(0) \ge 0$ and the exponential is always positive, it follows that $B(t) \ge 0$ for all $t \ge 0$.}
 
 \item[2/] 

Using the positivity of $B(\cdot)$, any solution of the first equation, with $(X_0,B_0,s_0,P_0)$ in $\r_{+}^{3}$ satisfies 
    $ \frac{dX(t)}{dt} \geq -K_{H} X(t)$ for all $t >0$. We deduce, by integration, that 
 $X(t)\geq X_0  e^{-K_Ht}$, for all $t >0$.
Therefore, for any vector $(X_0,B_0,s_0,P_0)$, with non-negative components, we have
 $ X(t)\geq 0$, for all $ t >0$.
\item[3/]
\textbf{4. Positivity of $s(t)$ via Hypothesis 3.} \\
The substrate equation reads
\[
\frac{ds}{dt} = - \Big(\frac{1}{Y_{B/s}} \mu(s) + m_s \Big) B + K_H X.
\]
By \textbf{Hypothesis 3}, $K_H X(t)$ is sufficiently large relative to $\mu(s)$, $m_s$, and $B(t)$,
which immediately gives
\[
\frac{ds}{dt} \ge 0 \quad \Rightarrow \quad s(t) \ge s_0 \ge 0, \quad \forall t \ge 0.
\]
$s(t)$ is also non-negative for all $t >0$, and cannot reach 0 in finite time, whatever $(X_0,B_0,s_0,P_0)$ in $\r_{+}^{3}$.
\item[4/]   By  {\bf Hypothesis 1},  {\bf Hypothesis 2}  and   the positivity of $B(\cdot)$,   all the coefficients of the last equation of Eq. \eqref{model s1} are non-negative.  We deduce that $\di\frac{dP}{dt} $  is non-negative and therefore $P(t)$ is increasing. Therefore, based on the fact that $P(0)$ is in $\r_+$, $P(t)$ is also non-negative for all $t >0$. \end{enumerate}
\end{proof}}

An equilibrium point of the non-linear system \eqref{model s1}, $E = (X^*, B^*, s^*,P^*)$, is a solution of the system
\begin{equation}\label{eqsHomogene}
   A(V)V=0 ,
\end{equation}
 where $A$ and $V$ are defined in Eq.\eqref{AetX}.

From the second equation of the system \eqref{eqsHomogene},
we derive either $B=0$ or $\mu (s)=k_d$.
\begin{itemize}
    \item[\checkmark] If $B =0$ then by the   first equation, we have $X=0$ whatever $s\in \r$, and the last two equations are automatically verified for any $s,P\in\r$. The vector $E=(0,0,s,P)$, for $s,P\in \r$ is then an equilibrium point.
    \item[\checkmark] If $B\neq 0$, 
    the  the fourth equation of Eq.\eqref{eqsHomogene} gives
    $
    \frac{k_d}{Y_{p/s}}=-m_p,
    $
    which is not possible because $m_p>0$ and $\frac{k_d}{Y_{p/s}}\geq 0$.
\end{itemize}
So, we have a continuum of equilibrium points given by $E=(0, 0, s, P)$, for $s$ and $P$ in $\r$, whatever the function $\mu$.
\\ 

Usually,   using the linearisation method, the stability of  equilibria of  ODEs is determined, by the sign of real part of eigenvalues of the Jacobian matrix.  For a point $E=(0,0,s^*,P^*)$ these eigenvalues are given by 
\begin{align}
    \lambda_1=-K_H,\qquad \lambda_2=\mu(s^*)-k_d,\qquad \lambda_3=0, \qquad \lambda_4=0.
\end{align}

  We cannot conclude on the stability or instability,  because  these equilibrium points are not hyperbolic (the eigenvalues $\lambda_3$ and $\lambda_4$ are zero). In
such situations, one could use the Lyapunov function (see \cite{walter}, page 319),
but it is generally difficult to find  such a function, especially for a complex non-linear system. We show in the next section, using Barbalat's lemma (see \cite{barbalat}), that the solutions of the system tend to an equilibrium point when $t$ tends to $+\infty$, which makes it a globally and asymptotically stable equilibrium.

Note that $\lambda_2\leq 0$ if and only if $\mu(s^*)\leq K_d$. In the study of basins of attraction, we will introduce the set
 of the values of $s$ such $\mu(s)\leq K_d$ defined by
\begin{align}
   \s:=\{s  \in \r_+ : \mu(s)\leq k_d\}.
\end{align}
Under {\bf Hypothesis 1} and {\bf Hypothesis 2}, this set has a non-empty interior.
For the Monod  law, we have
\begin{align}\label{DefS}
   \s=[0,\lambda] \>\>\>\>\>\>  \mbox{with}\>\>\>\>\>\> \lambda=\frac{k_d k_s}{\mu_{max}-k_d}.
\end{align}

\subsection{Asymptotic behaviour}
In this section, we describe the asymptotic behavior of the system solution \eqref{model s1} and   deduce the stability of the equilibrium points. 
We begin by showing that when time goes to infinity, the solution tends to a point of the form $(0, 0, s^*, P^*)$ for some values $s^*$ and $P^*$ which depend on the parameters of the problem and the initial conditions.

\begin{theorem}\label{pr1}
Let $(X_0,B_0,s_0, P_0)$ be a vector with non-negative components, and suppose  {\bf{ Hypothesis 1}} and {\bf{ Hypothesis 2}} are satisfied.
Then the solutions of the system \eqref{model s1} with the initial condition
$(X(0), B(0), s(0), P(0)) = (X_0, B_0, s_0, P_0) $
converge asymptotically to an equilibrium point $(0, 0, s^*, P^*)$ such that
\begin{eqnarray}
s^{\star}&\leq& s_{0}+\frac{\left(1+\di\frac{m_s Y_{B/s} }{k_d}\right)X_{0}+\alpha B_0}{\alpha Y_{B/s}}\label{s00}\\
P^{*}&=&P_0+a B_0 +b (X_0+s_0-s^*).\label{pcv}
\end{eqnarray}
The coefficients $a$ and $b$ are positive numbers given by the following expressions:
\begin{eqnarray*}
    a&=&\frac{Y_{P/s} m_P+k_d Y_{B/s}\left(\alpha-\frac{m_s}{k_d}\right)}{k_d Y_{P/s}\left[1-Y_{B/s}\left(\alpha-\frac{m_s}{k_d}\right)\right]}\\
    b&=&\frac{Y_{B/s}\left(Y_{P/s} m_P+k_d\right) }{k_d Y_{P/s}\left[1-Y_{B/s}\left(\alpha-\frac{m_s}{k_d}\right)\right]}.
\end{eqnarray*}
Moreover, when $X_0$ or $s_0$ is non-zero, we have $s^{\star}>0$ and $s^{\star}$  belongs to $\s$, where $\s$ is defined by Eq. \eqref{DefS}.
\end{theorem}
%
\begin{proof}
Consider the  
function $Z(t)$ defined by
\begin{align}\label{Z}
    Z(t)= \left(1+\frac{m_sY_{B/s} }{k_d}\right)X(t)+\alpha B(t)+\alpha Y_{B/s} s(t).
\end{align}
This function satisfies some properties that are essential for the rest of the proof. We start by mentioning them before evaluating the limits at infinity of the components of the solution of the system \eqref{model s1}.
\begin{itemize}
    \item $Z(t)\geq 0$, for all $t\geq 0$.
    \item The derivative of $Z(t) $ is given by
\begin{equation}\label{dZ}
    \frac{dZ}{dt}=K_H \left[\left(\alpha Y_{B/s} -1\right)-\frac{m_sY_{B/s} }{k_d}\right]X,
\end{equation}
which makes it possible to write $Z(t)$ also in the form
\begin{equation}\label{Zt2}
    Z(t)=Z(0)+K_H \left[\left(\alpha Y_{B/s} -1\right)-\frac{m_sY_{B/s} }{k_d}\right]\int^{t}_{0}X(\tau)d\tau
\end{equation}
\item There exists $0\leq Z\infty <+\infty$ such that
\begin{equation}\label{LimZt}
    \lim_{t \to +\infty}Z(t) := Z\infty, 
\end{equation}
indeed $Z(t)$ is decreasing (by {\bf Hypothesis 2}, 
$\left(\alpha Y_{B/s} -1\right) -\dfrac{m_sY_{B/s} }{k_d})<0$) and minorized by zezo. 
\item The second derivative of $Z(t)$ is given by
\begin{equation}\label{DerveSecond}
    \frac{d^{2}Z}{dt^{2}}=K_H \left((\alpha Y_{B/s} -1)-\frac{m_sY_{B/s} }{k_d}\right)
    \left(-K_H X+\alpha k_d B\right).
\end{equation}
\end{itemize}
\begin{enumerate}
    \item We begin by showing that the solution of Eq. \eqref{model s1}
    converge asymptotically to an equilibrium point. \\ 
    
\noindent\underline{\sc Limits of $X(t)$ and $B(t)$}\\
With respect to the definition of $Z(t)$, the limit \eqref{LimZt} implies that the variables $X(t)$, $B(t)$ and $s(t)$ are bounded (since they are
non-negative).  We deduce that the
second derivative \eqref{DerveSecond} is bounded, therefore $ \frac{dZ}{dt}$ is uniformly
continuous on $\r^+$. Barbalat's Lemma see \cite{barbalat} allows us to assert that 
  $  \underset{t \rightarrow +\infty} {\lim} \frac{dZ}{dt}=0$.
By Eq. \eqref{dZ}, we obtain
  $\lim_{t\to+\infty} X(t)=0$. \\ 
Similarly, $\frac{d^{2}X}{dt^{2}}$, which is expressed as a function of $X(t)$ and $B(t)$, is also bounded, and then $\frac{dX}{dt}$ is uniformly continuous and by Barbalat's lemma, we deduce
  $   \underset{t \rightarrow +\infty} {\lim} \frac{dX}{dt}=0$,
and therefore, by the first equation of the system and knowing that $\di \lim_{t\to+\infty} X(t)=0$, we have
  $  \underset{t \rightarrow +\infty} {\lim} B(t)=0$. 
  \\
  
 %
\noindent\underline{\sc Limit of $s(t)$}\\
By Eqs. \eqref{Z}, and \eqref{LimZt}, and as we have already shown that $\di\lim_{t\to+\infty}X(t)=\lim_{t\to+\infty} B(t)=0$, we can affirm that there exists a positive real $s^{\star}$ such that
\begin{equation}
   \lim_{t \to +\infty} s(t)=s^{\star}=\frac{Z_{\infty}}{\alpha Y_{B/s}}.
\end{equation}
The function $Z$ being decreasing, we have $Z_{\infty} \leq Z(0)$, and then $s^{\star}\leq \di\frac{Z(0)}{\alpha Y_{B/s}}$.\\
As $Z(0) =\left(1+\di\frac{m_s Y_{B/s} }{k_d}\right)X_{0}+\alpha B_{0}+\alpha Y_{B/s}s_0$, we get Eq. \eqref{s00}.\\

\noindent\underline{\sc Limit of $P(t)$}\\
Consider $Z(t)$ written in the form of Equation \eqref{Zt2}.  The function $Z (.)$ is bounded, we deduce that it is the same for $\int^{t}_{0}X(\tau)d\tau$, then using the integration of the system \eqref{model s1} between $0$ and $t$ for $t>0$ gives 
\begin{eqnarray*} 
X(t) &=&  X(0)-K_{H}\int_0^tX(\tau)d\tau + \alpha\, k_{d}\int_0^tB(\tau)d\tau\\
    B(t) &=& B(0) +\int_0^t\mu(s(\tau))B(\tau)d\tau - k_{d}\int_0^tB(\tau)d\tau\\
  s(t)	&=& s(0) - \di\frac{1}{Y_{B/s}}\int_0^t\mu(s(\tau))B(\tau)d\tau   - m_{s}\int_0^tB(\tau)d\tau  + K_{H}\int_0^tX(\tau)d\tau \\
  P(t)&=&P(0)+\int_0^t \frac1{Y_{P/s}}\mu(s(\tau))B(\tau)d\tau+m_P\int_0^t B(\tau)d\tau.
\end{eqnarray*}
We deduce by cascade that 
 $\di \int^{t}_{0}B(\tau)d\tau < +\infty$,  $\di    \int^{t}_{0} \mu(s(\tau))B(\tau)d\tau<+\infty$  and that $\di\lim_{t\to+\infty}P(t)$ is finite and positive as a limit of a positive and bounded increasing function.
 
%
Taking the limit in the last expressions of $X(t)$, $B(t)$ and $P(t)$, when $t$ tends to $+\infty$, and 
having $\di\lim_{t\to+\infty} X(t)=\di\lim_{t\to+\infty} B(t)=0$ and Eq. \eqref{s00} gives 
\begin{eqnarray}\label{e0}
  0&=& X_0-K_H\int^{+\infty}_{0}X(\tau)d\tau+\alpha k_d \int^{+\infty}_{0} B(\tau)d\tau \nonumber \\
0 &=& B_0-k_d\int^{+\infty}_{0}B(\tau)d\tau+\int^{+\infty}_{0} \mu(s(\tau))B(\tau)d\tau \nonumber\\
s^*	&=& s_0 - \di\frac{1}{Y_{B/s}}\int_0^{+\infty}\mu(s(\tau))B(\tau)d\tau  \nonumber\\
&&\qquad\qquad\qquad- m_{s}\int_0^{+\infty}B(\tau)d\tau  + K_{H}\int_0^{+\infty}X(\tau)d\tau\nonumber \\
\lim_{t\to+\infty}P(t)&=&P_0+\int_0^{+\infty} \frac1{Y_{P/s}}\mu(s(\tau))B(\tau)d\tau+m_P\int_0^{+\infty} B(\tau)d.\tau\label{limPinfini}
    \end{eqnarray} 

Linear combinations of the first three equations give

\begin{equation}\label{intb}
   \int^{t}_{0}B(\tau)d\tau=\frac{B_0+ Y_{B/s}\left[(X_0+s_0-s^*)\right]}{k_d\left[1-Y_{B/s}(\alpha-\frac{m_s}{k_d})\right]} 
\end{equation}
\begin{equation}\label{intmubb}
  \int^{t}_{0} \mu(s(t))B(\tau)d\tau=\frac{Y_{B/s}\left[B_0(\alpha-\frac{m_s}{k_d})+(X_0+s_0-s^*)\right]}{\left[1-Y_{B/s}(\alpha-\frac{m_s}{k_d})\right]}  
\end{equation}
By replacing the expressions \eqref{intb} and \eqref{intmubb} in Eq. \eqref{limPinfini}, we obtain the expression \eqref{pcv}. Note that by {\bf Hypothesis 2}, $\left[1-Y_{B/s}(\alpha-\dfrac{m_s}{k_d})\right]$ is strictly positive.\\

\item Let us now show that $s^{\star}$ cannot be equal to 0 when $X_0$ or $s_0$ are not zero. Suppose that $X_0\neq0 $ or $s_0\neq0 $ and that $s(t)$ tends to 0 when t tends to infinity. Then by continuity, the function $\mu(s(t))$ will tend towards zero too, and there would exist $T$
such that for all $t\geq T$ we have
\begin{equation}\label{inimu2}
    \mu(s(t))\leq Y_{B/s}\alpha k_d- Y_{B/s} m_s.
    \end{equation}
The sum of the first and third equations of the system \eqref{model s1} allows us to write
\begin{equation}
    \frac{d}{dt}(X(t)+s(t))=\frac{1}{Y_{B/s}}(Y_{B/s}\alpha k_d- Y_{B/s} m_s -\mu(s))B.
\end{equation}
By \eqref{inimu2} we conclude that 
  $\dfrac{d}{dt}(X(t)+s(t))\geq 0$ for all $ t>T.$ 
 Using the same argument as the one used in the demonstration of the positivity of the solution,   we conclude that if $X(0)+s(0)>0$, the variable $X(\cdot)$
or $s(\cdot)$ cannot reach 0 in finite time. We  then have
   $X(t)+s(t)\geq X(T)+s(T)>0$ for all $t>T$, 
which is a contradiction to the fact that $X(t)+s(t)$ tends to $0$ when $t$ tends to $+\infty$.\\

\item We prove now that $s^*$ belongs to $\s$.
Let $(X_0,s_0,B_0,P_0)$ be in $\r_{+}^{4}$ with $X_0 > 0$ and $B_0 >0$.  Suppose that 
$s^* $, defined by Equation  \eqref{s00}, does not belong to $\s$. Then there would exist $\tilde{T} > 0$
such that for all $t>\tilde{T}$ we would have :
  $ \mu(s(t))-k_d>\eta:=\frac{\mu(s^*)-k_d}{2}.$
By the system \eqref{model s1}, we would then have
$\frac{dB(t)}{dt} > \eta B(t).$ Therefore, we would have $
    B(t)>B(s)\exp^{\eta(t-s)} $ for all $t,s > \tilde{T}, \> \text{such that} \> \> t\geq s,$ and then $B(.)$ cannot converge asymptotically to $0$, which is incompatible with  the previous results .
    \end{enumerate}
    \end{proof}

\subsection{Properties of the equilibrium points}
As we have already pointed out,  the equilibria of the system \eqref{model s1} are not hyperbolic. 
We cannot conclude on their stability by directly using the linearization technique and the central variety theorem (see \cite{PERCO}). However,  using a suitable variable change, we prove the existence of an invariant stable variety in the following theorem.

%
\begin{theorem} \label{pr3}

Assume that {\bf Assumption 1} and {\bf Assumption 2} are satisfied. For each stationary     point $ E = (0,0,s^*,P^*)$ with $s^*$ in  $\mbox{int } \s$, 
there exists a  stable two-dimensional invariant variety $\mathcal{M}$ in $\r_{+}^{4}$ such that any solution of \eqref{model s1} with the initial condition in $\mathcal{M}$ converges asymptotically to E.
\end{theorem}
\begin{proof}
 The proof of this theorem is based on two variables changes, the first is for the variable $s$, and the second is for the variable $P$. We obtain
a new system associated with the system \eqref{model s1} and whose equilibria are now hyperbolic.
To define 
the variable change for $s$,  we fix $s^* > 0$ and $P^*>0$ so that $\mu(s^*) < k_d$. 
Consider the solutions of \eqref{model s1} with $B(t)> 0$, 
 for $t \in [0,+\infty)$, and let
\begin{equation}\label{0}
    z(t)=\dfrac{X(t)+(s(t)-s^*)}{B(t)}+\varphi,\quad 
\text{with} \quad 
    \varphi:=\dfrac{\dfrac{\mu(s^*)}{Y_{B/s}}-(\alpha k_d -m_s)}{\mu(s^*)-k_d},
\end{equation}
  which is equivalent to
\begin{equation}\label{d}
    s(t)=s^*-X(t)+(z(t)-\varphi)B(t).
\end{equation}
The function $z(t)$ satisfies the equation (see the Appendix,  page
\pageref{page: EquationZ}, for more details
)
\begin{equation}\label{c}
     \frac{dz}{dt} =-\gamma(\mu(s)-\mu(s^*)) -(\mu(s)-k_d)z
 \end{equation}
  with
   $\di  \gamma:=\left[\frac{k_d (1-\alpha Y_{B/s})+Y_{B/s} m_s}{Y_{B/s}(k_d-\mu(s^*))}\right] > 0$,  and using Eq. \eqref{d} we can write $\mu(s)$ in terms of the new variable and denote by $F(z,X,B)$ the following expression
   \begin{equation}\label{gzxb}
    F(z,X,B) := \mu\biggl(s^*-X+(z-\varphi)B\biggr)
\end{equation}
The second change of variable is defined by setting $P^* > 0$ and $s^* > 0$ such that $\mu(s^*) < k_d $. For the solutions of \eqref{model s1} with $B(t)> 0$ for $t \in [0,+\infty)$, we define
\begin{equation}\label{0P}
    W(t)=\frac{P(t)-P^*}{B(t)}+\omega,\quad 
\text{with} \quad 
    \omega:=\dfrac{\dfrac{-\mu(s^*)}{Y_{P/s}}-m_P}{\mu(s^*)-k_d},
\end{equation}
 which is equivalent to
\begin{equation}\label{dP}
    P(t)=P^*+(W(t)-\omega)B(t). 
\end{equation}
 The variable $W(t) $ satisfies the equation (see \ref{secA1}, page \pageref{page:Calcul2}, for more  details) 
 \begin{equation}\label{cP}
      \frac{dW}{dt} =\phi(\mu(s)-\mu(s^*)) -(\mu(s)-k_d)W,   
 \end{equation}
 with
   $\di  \phi:=\left[\frac{k_d +Y_{P/s} m_P}{Y_{P/s}(k_d-\mu(s^*))}\right] > 0$ .\\
The system \eqref{model s1} becomes with the variables $z$ and $W$, using Equation \eqref{gzxb}

 \begin{equation}\label{xbzw}
\left\{ 
    \begin{array}{rcll}
    \displaystyle
    \frac{dz}{dt} &=&  -\gamma \biggl(F(z,X,B) - \mu(s^* )\biggr)-\biggl(F(z,X,B) -k_d\biggr)z,
    \vspace{0.1cm}\\
    \displaystyle
   \frac{dX}{dt} &=& -K_H X+\alpha k_d B ,
   \vspace{0.1cm}\\
   \displaystyle
  \frac{dB}{dt}	&=&  \biggl(F(z,X,B) -k_d\biggr) B,\vspace{0.1cm}\\
  \displaystyle
\frac{dW}{dt} &=&  \phi \biggl(F(z,X,B) - \mu(s^* )\biggr)-\biggl(F(z,X,B) -k_d\biggr)w,\\
     \end{array} \right.
  \end{equation} 
     defined 
  in the domain     given  by
\begin{eqnarray*}
    \mathcal{D}&:=&\{(z,X,B,W) \in \r \times \r_{+} \times \r_{+}^{*}\times \r\quad\text{such that}\\
    &&\qquad s^*-X+(z-\varphi)B \geq 0  
    \quad\text{and}\quad P^*+(W-\omega)B \geq 0 \}.
\end{eqnarray*}

 By condition $s^* \in int \> \s $, we have  
$\mu(s^*)\neq k_d$, and we can check that $(0,0,0,0)$ is the only equilibrium of \eqref{xbzw} in $\mathcal{D}$.
 Indeed, if $E = (z^*, X^*, B^*,W^*)$ is an equilibrium point, from 
the $3^{th}$ equation, 
we derive that either $B=0$ or $F(z^*,X^*,B^*)=k_d$.
\begin{itemize}
 
\item If $F(z^*,X^*,B^*)=k_d$, from the equations $1^{th}$ and $4^{th}$ of the system \eqref{xbzw} we deduce   that $\mu (s^*)=k_d$ contradicts condition $s^* \in int \> \s$.
\item If $B =0$, by the equation $2^{th}$, we have $X=0$ whatever $s^* \in \r_+$. In this case, from the first equation we deduce that we necessarily have
  $z=0$ since we cannot have $\mu (s^*)=k_d$. 
  \item The same reasoning as used for the equation $4^{th}$ gives $W=0$. 
\end{itemize}
Thus, the vector $E=(0,0,0,0)$ is the only equilibrium point of the system \eqref{xbzw} in $\mathcal{D}$ and the Jacobian matrix associated 
is given by
\begin{align}\label{j0}
   \mathfrak{J}_{(0,0,0,0)}=\begin{pmatrix}
-(\mu(s^*)-k_d) & \gamma\mu^{'}(s^*) & \gamma \varphi\mu^{'}(s^*) & 0 \\ 0 &   -K_H & \alpha k_d & 0 \\ 0 &  0 &
\mu(s^*)  - k_{d} & 0 \\
0 & -\phi\mu^{'}(s^*) & -\phi \varphi\mu^{'}(s^*) &-(\mu(s^*)-k_d)   
\end{pmatrix}
\end{align}
It admits a double eigenvalue $r_1$ and two single eigenvalues $r_2$ and $r_3$:
\begin{equation}\label{vp}
     r_1=-(\mu(s^*)-k_d), \qquad
  r_2=-K_H, \qquad r_3=\mu(s^*) - k_{d}, 
\end{equation}
 all of them are real and non-zero.  According to the Central Variety Theorem 
\cite{Wiggins} page 35) and under the condition $\mu(s^*)< K_d$, the point $(0,0,0,0)$ is thus a hyperbolic equilibrium point with a stable two-dimensional variety $\mathcal{L}$ and an unstable two-dimensional variety $\mathcal{U}$ which are, respectively, positive and negative invariants. So, any trajectory $(z(.),X(.),B(.),W(.))$ in $\mathcal{D}$ of the system \eqref{xbzw} in the stable two-dimensional invariant variety $
\mathcal{L} \cap \mathcal{D} $  converges asymptotically to $(0,0,0,0)$.

Finally, we conclude from the Invariance of Stability (see \cite{BOOK DIFF}, Proposition 6.6, page 283) that
by equivalence there exists a stable two-dimensional invariant variety $\mathcal{M}$ in $\r_{+}^{4}$ such that any solution of the system \eqref{model s1} with the initial condition in $\mathcal{M}$ converges asymptotically to $(0,0,s^*, P^*)$.
\end{proof}

\section{ Numerical tests}
 This section is dedicated to different numerical tests. As a validation of our model, we compare the results obtained by our system  \eqref{model s1} with those obtained in  \cite{ma 2013} and \cite{p2006p}.  After that, we present numerical tests that confirm the theoretical results obtained in the previous sections.
The system \eqref{model s1} is autonomous. To achieve good accuracy, we choose the Runge-Kutta method of order 4 (RK4) and the Matlab computer software package to solve it.\\

\noindent\underline{\bf Validation tests}

\noindent For the first test, 
we  consider the data of Table \ref{t1} and Table \ref{t2} given in reference \cite{ma 2013} for the first table and references \cite{p2006p} and \cite{saleh} for the second.
\begin{figure}[h]
   \begin{center}
    \includegraphics[width=1.9in] 
      {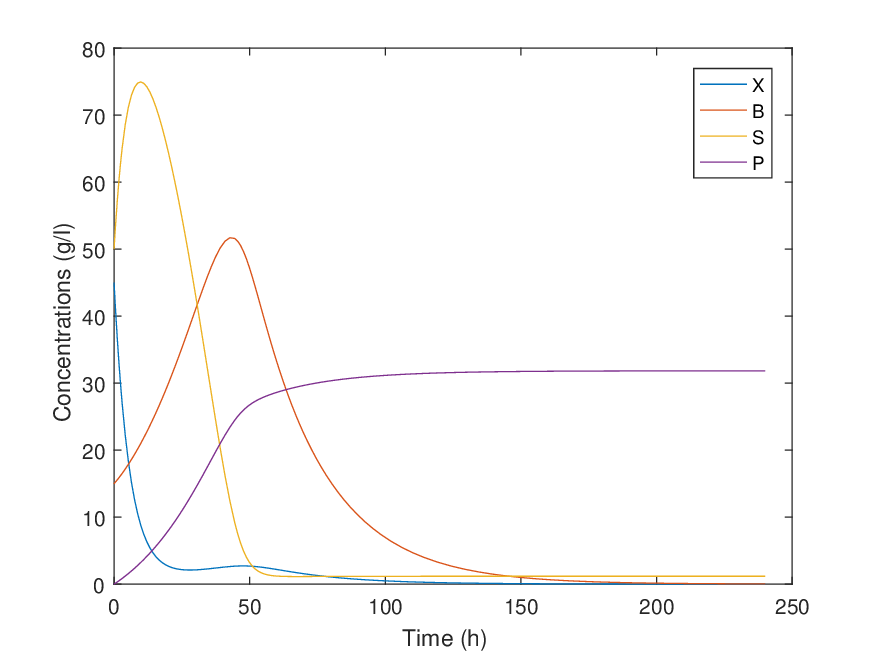}
    \caption{Evolution of $X$, $B$, $s$ and $P$, using  the parameters of Table \ref{t1}, Table \ref{t2} and the initial conditions $(X_0,B_0,s_0,P_0)=(45,15,50,0)$}
    \label{f0}
   \end{center}
\end{figure}
\begin{table}[h]\begin{center}
\caption{}\label{t1}%
\begin{tabular}{@{}cccccccc@{}} \toprule
 $\mu_{max}$ ($1/h$) &$ k_s$ (g/l) &$ k_d$ (1/h)& $Y_{B/s}$  (g/g) &$ m_s $  (1/h) & $s_0$ (g/l)& $B_0$ (g/l) &$ P_0$ (g/l) \\ \midrule
   0.096 & 11.27  & 0.048 & 1.19 & 0.0047 & 50 & 15 & 0  \\
 \botrule
\end{tabular}\end{center}
\end{table}
\begin{table}[h]\begin{center}
\caption{}\label{t2}%
\begin{tabular}{@{}ccccc@{}} \toprule
   $\alpha$ & $k_H $  & $m_P$ (1/h) & $\frac{1}{Y_{P/s}}$ (g/g) &  $X_0$ (g/l)\\ \midrule
  0.2 & 0.176  & 0.002 & 0.2 & 45  \\
 \botrule
\end{tabular}\end{center}
\end{table}
\noindent The evolution curves over time of the organic matter $X$, the substrate $s$, the biomass $B$ and the product $P$, using the resolution of the system \eqref{model s1} by Matlab are given in Fig. \ref{f0}. We notice that the curve of the substrate $s$ is very close to that given experimentally in \cite{ma 2013} (Fig. 3, page 195). This is a validation of our model.
The curve of the biomass $B$ of the same reference in which the results are also experimental converges towards a stationary phase, contrary to our system where it tends towards 0. This is explained by the fact that the experiment in \cite{ma 2013} was carried out with fed-batch fermentation, while in the rhizosphere, the culture is discontinuous (batch) and forces the cell to go through exponential, stationary and finally decay phases.

\noindent After this first validation test, we consider  another validation test related to the formation of the product, using the  data of Table \ref{t3} and Table \ref{t4}  given in the references \cite{p2006p}, 
\cite{ma 2013} and  \cite{saleh}.
  In \cite{p2006p}   the authors highlighted the influence of various carbon substrates in the production of the cellulase protein using T. reesei 97.177 and Tm3. Cellulase shows the maximum yield of cellulose as a synthetic source.
Kinetic studies were also done for growth and production using the Monod equation and Leudeking Piret model, respectively.
Our results 
are given in Fig. \ref{f1} and  
are close to those of \cite{p2006p} (Fig. 5, page 1879).
\begin{figure}[h]
    \centering
     \includegraphics[width=1.9in]{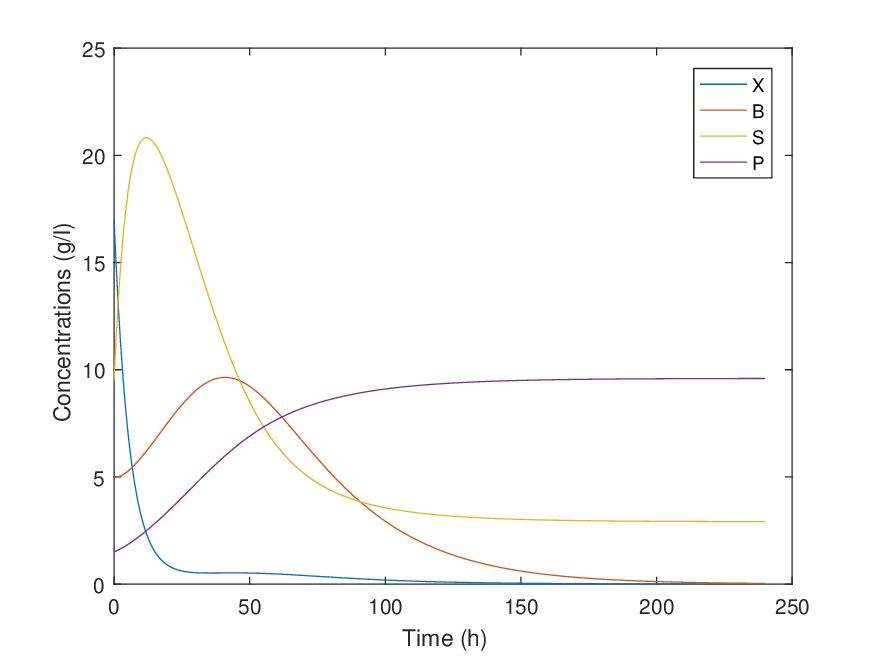}
    \caption{Evolution of $X$, $B$, $s$, $P$, using the parameters of  Table \ref{t3}  and  Table \ref{t4}
with  $(X_0,B_0,s_0,P_0)=(17,5,9.5,1.5)$}
    \label{f1}
\end{figure}
\begin{table}[h!]
\begin{center}
\caption{}\label{t3}%
\begin{tabular}{@{}ccccccc@{}} \toprule
  $\mu_{max}$ (1/h) & $k_s$ (g/l) & $\frac{1}{Y_{P/s}}$ (g/g) & $m_P$ (1/h)& $B_0$ (g/l)& $s_0$ (g/l) & $P_0$ (g/l) \\ \midrule
   0.2 & 35.55 &  0.2 & 0.002 & 5 & 9.5 & 1.5 \\
 \botrule
\end{tabular}\end{center}
\end{table}
\begin{table}[h!]\begin{center}
\caption{}\label{t4}%
\begin{tabular}{@{}cccc@{}} \toprule $\alpha$ & $k_H $ & $Y_{B/s}$ (g/g) & $m_s$ (1/h) \\ \midrule
   0.2 & 0.176 &  1.19 & 0.0047   \\
 \botrule
\end{tabular}   
\end{center}
\end{table}

\vspace{0.3cm}

\noindent\underline{\bf Numerical experiments}\\

\noindent \underline{\it First test: variation of $X_0$}. In Theorem \ref{pr1} we studied the asymptotic behaviour of the system \eqref{model s1}. We have highlighted, theoretically, the importance of the initial condition of organic matter in the production of enzymes (cellulase). Indeed we have shown that when time tends to infinity, the solution tends to an equilibrium point 
$(0,0,s^*, P^*)$ for specific values $s^*$ and $P ^*$, which depend on the parameters of the problem and the initial conditions.
  In the following test,  these theoretical considerations are confirmed numerically.\\
 We fix all the parameters and vary the initial condition of the organic matter $X_0$. The solution of the system \eqref{model s1} is
performed with the data of  Table \ref{t1}  and   Table \ref{t2}
for initial conditions: 
\begin{align}
    X^{1}_{0}=45 \quad; \quad X^{2}_{0}=90 \quad; \quad X^{3}_{0}=180 \quad; \quad X^{4}_{0}=360.
\end{align}
\begin{figure}[h!]
    \begin{minipage}[c]{.5\linewidth}
    \centering
        \includegraphics[width=.8\textwidth]{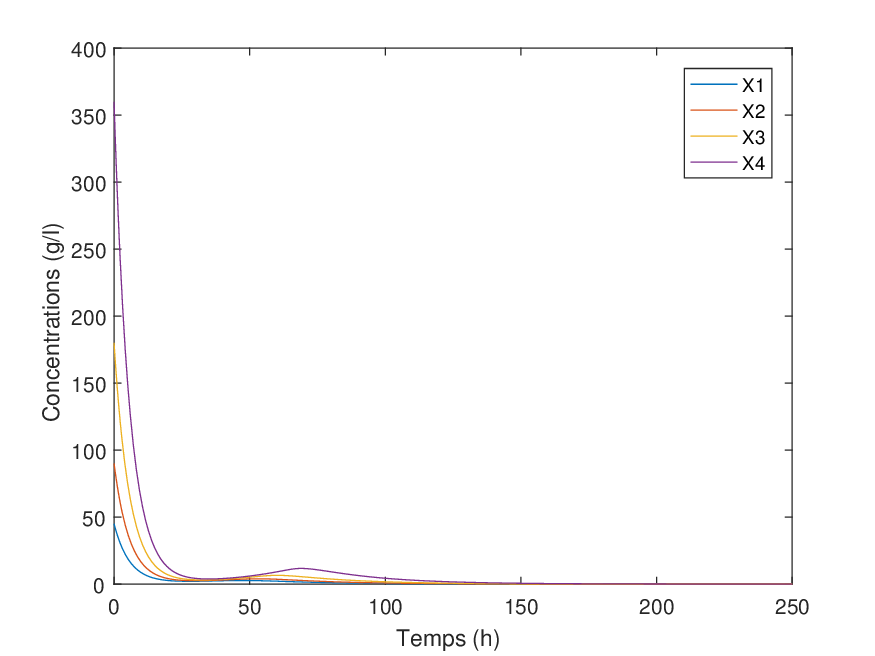}
    \caption{ Evolution of $X$ for\\ initial conditions
    $X^{i}_{0}$ for $i = 1,...,4$ .}
    \label{fx01}
    \end{minipage}
    \hfill%
    \begin{minipage}[c]{.47\linewidth} 
    \centering
        \includegraphics[width=.8\textwidth]{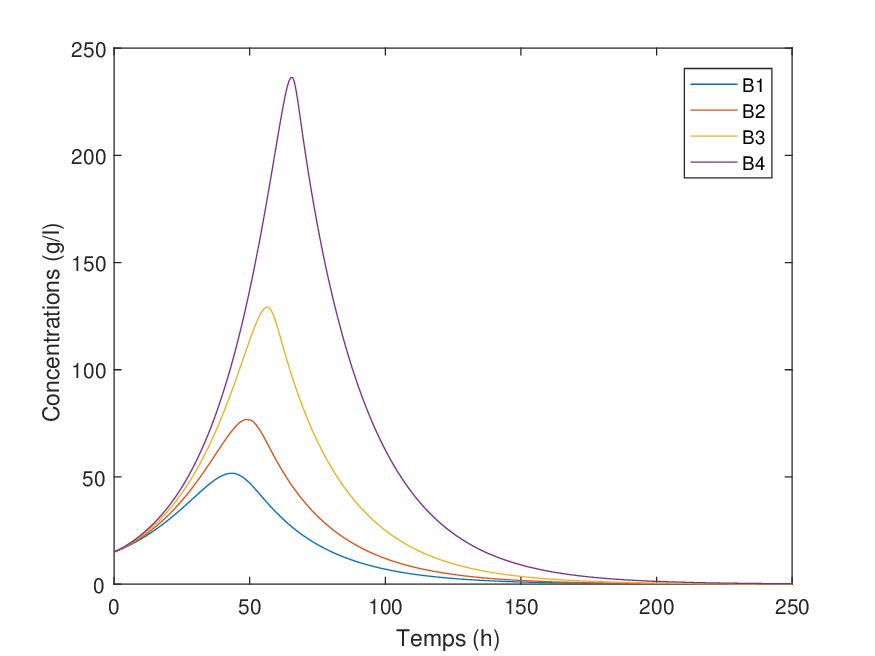}
    \caption{  Evolution of $B$ for\\ initial conditions  $X^{i}_{0}$ for $i = 1,...,4.$ 
}
    \label{fbx}
    \end{minipage}
\hfill%
\begin{minipage}[c]{.47\linewidth} 
\centering
        \includegraphics[width=.8\textwidth]{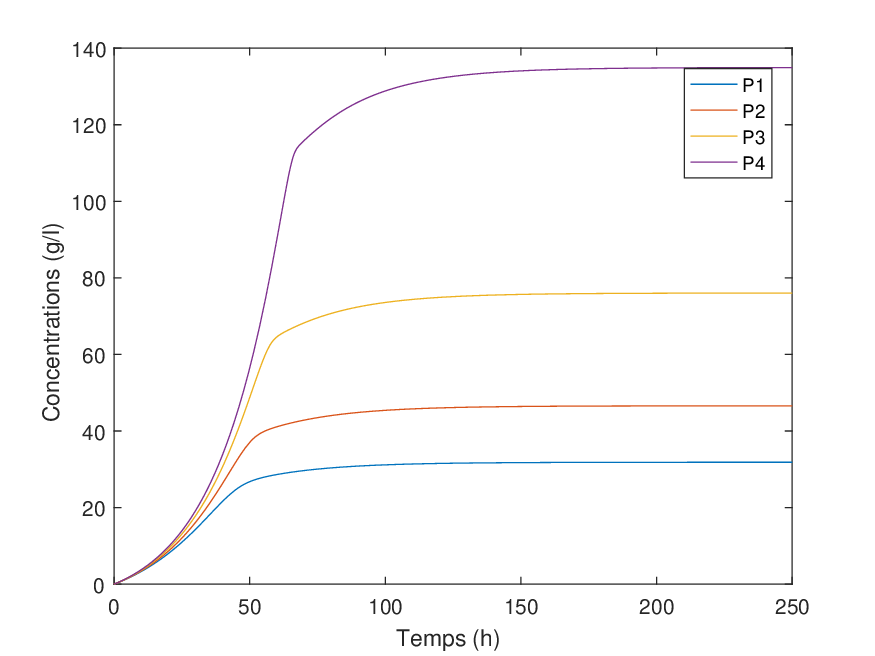}
    \caption{ Evolution of $P$ for \\ initial conditions $X^{i}_{0}$ for $i = 1,...,4.$
}
    \label{fp01}
     \end{minipage}
    \hfill%
    \begin{minipage}[c]{.47\linewidth} 
    \centering
        \includegraphics[width=.8\textwidth]{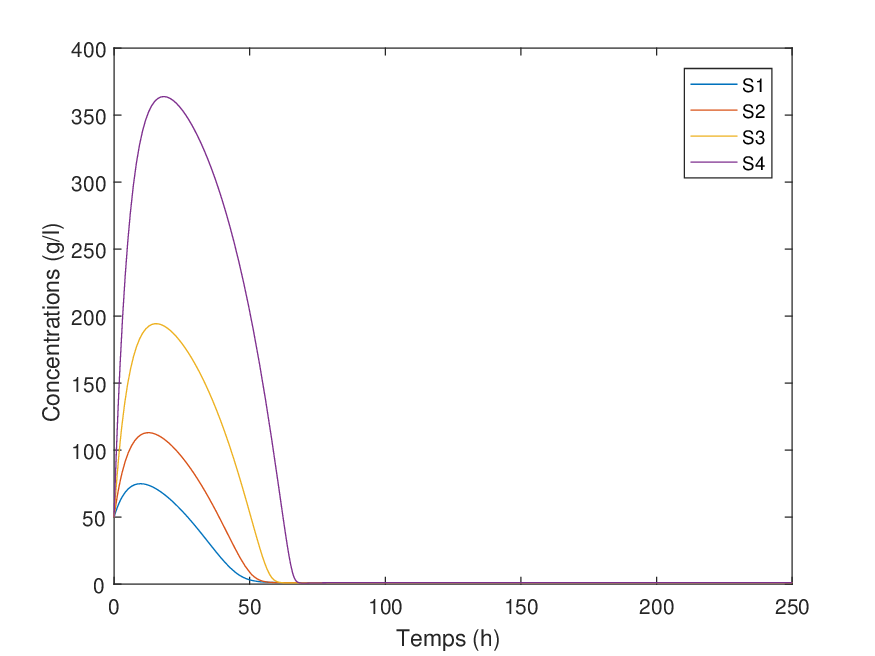}
    \caption{Evolution of s(.) for \\initial conditions $X^{i}_{0}$ for $i = 1,...,4.$
}
    \label{fsx}
     \end{minipage}
\end{figure}
The simulations in Figs. \ref{fx01}, \ref{fbx}, \ref{fsx} and \ref{fp01} describe the evolution, in the same time interval, of the variable $X_i$, $ B_i$, $s_i$  and $P_i$, respectively associated with the initial conditions $X^{i}_{0}$ for $i = 1,...,4$.
We notice as 
        expected  in Theorem \ref{pr1} , that: 
        \begin{itemize}
            \item[-] if we have more initial organic matter, we will have more product formation, 
            \item[-] whatever the initial condition, the variables $X$ and $B$ tend to zero when time becomes large,
            \item[-] for an initial condition  $B_0 >0$, the variable $s(t)$ do not tend towards 0 but tends towards a value $s^*>0$, when time becomes large.
        \end{itemize} 
  The limit values are summarised in Table \ref{t5}.\\ 
   \begin{table}[h]
   \begin{center}
\caption{Limit values of $X$, $B$, $s$ and $P$ for different initial conditions $X_0$}\label{t5}
\begin{tabular}{@{}ccccc@{}} \toprule $X_0$ & $ X^* $ & $B^*$ & $s^*$ & $P^*$  \\  \midrule
   45 & 2,7402e-07 & 8,7454e-10  &1,1745  & 31,8399 \\  \midrule
   90 & 7,8690e-07 & 9,1595e-10  & 1,1761 &  46,5702\\ \midrule
   180 & 7,1684e-07 &  8,8018e-10 & 1,1766 & 76,0315  \\ \midrule
   360 & 2,3503e-07 &  8,3535e-10 & 1,1766 & 134,9544  \\\botrule
\end{tabular}\end{center}
\end{table}
\noindent \underline{\it Second test: variation of $k_d$}. 
We now consider another test that allows us to observe
the effect of the variation of $k_d$ on the evolution of the concentrations of $X$, $B$, $s$, and $P$.   In this test, we consider the resolution of the system \eqref{model s1} with the data of Table~1 and Table~2 except for $\mumax$ which will take the value. 
$\mumax=0.2$, and we make vary the value of $k_d$ as 
following 
\begin{equation}
    k_d^{1}=0.03 \qquad; \qquad k_d^{2}=0.09 \qquad; \qquad k_d^{3}=0.12 \qquad; \qquad k_d^{4}=0.18.
\end{equation}
The simulations in Figs. \ref{kd0.03}, \ref{kd0.09}, \ref{kd0.12} and  \ref{kd0.18} represent the evolution, over time, of the variables
$X$, $B$, $s$ and $P$, respectively associated with the values of $k_d^{i}$ for $i=1,...4.$ The respective curves of $X$, $B$, $s$ and $P$ are presented in the same time interval to visualize their relative behaviours. As modelled in the equations, we notice, as expected, that 
when the value of $k_d$ increases: \\
\begin{figure}[h!]
    \begin{minipage}[c]{.47\linewidth}
    \centering
        \includegraphics[width=.8\textwidth]{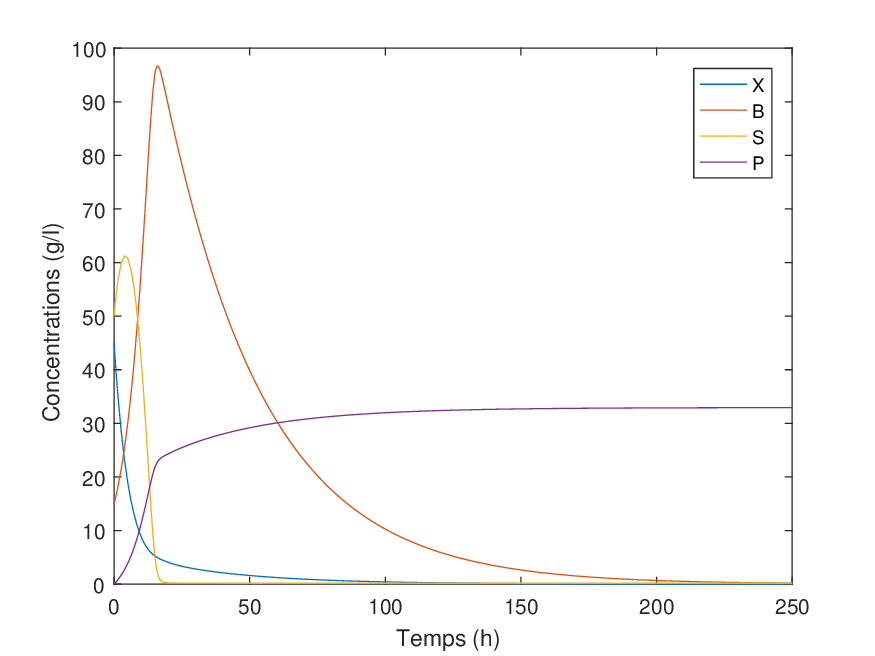}
    \caption{Evolution of $X(\cdot)$,\\ $B(\cdot)$, $s(\cdot)$ and  $P(\cdot)$ with $k_d=0.03$}
    \label{kd0.03}
    \end{minipage}
    \hfill%
    \begin{minipage}[c]{.47\linewidth} 
       \centering
        \includegraphics[width=.8\textwidth]{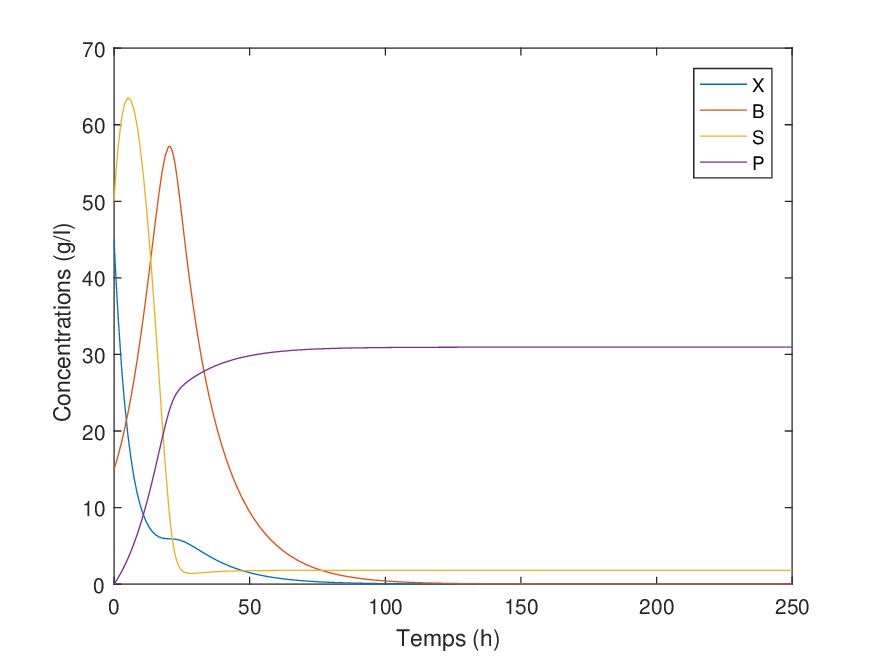}
    \caption{Evolution of $X(\cdot)$,\\ $B(\cdot)$, $s(\cdot)$ and  $P(\cdot)$  with $k_d=0.09$}
    \label{kd0.09}
    \end{minipage}
\hfill%
\begin{minipage}[c]{.47\linewidth} 
        \centering
        \includegraphics[width=.8\textwidth]{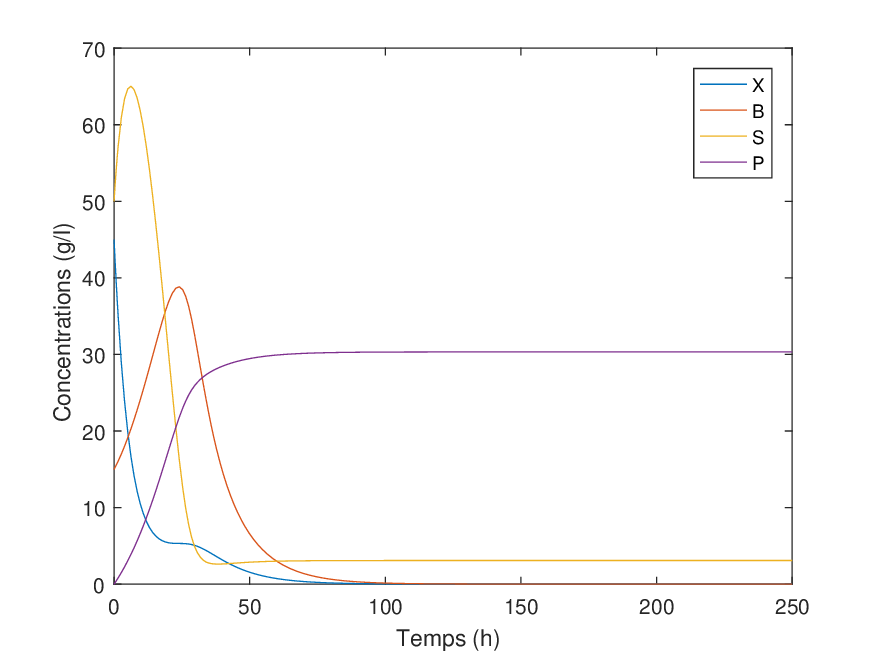}
    \caption{ Evolution of $X(\cdot)$,\\ $B(\cdot)$, $s(\cdot)$ and  $P(\cdot)$ with $k_d=0.12$ }
    \label{kd0.12}
     \end{minipage}
    \hfill%
    \begin{minipage}[c]{.47\linewidth}
        \centering
        \includegraphics[width=.8\textwidth]{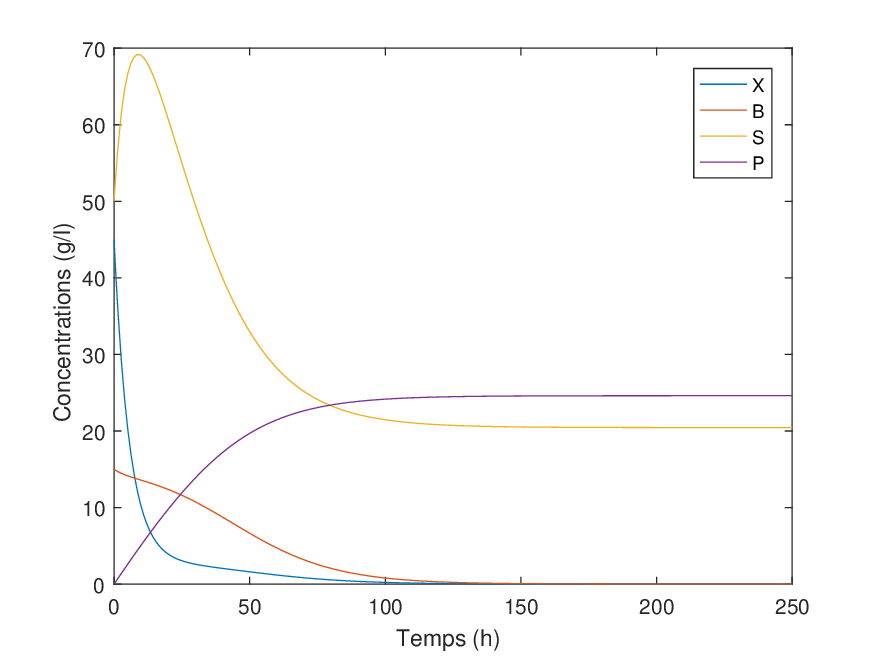}
    \caption{ Evolution of $X(\cdot)$,\\ $B(\cdot)$, $s(\cdot)$ and  $P(\cdot)$  with $k_d=0.18$}
    \label{kd0.18}
     \end{minipage}
\end{figure}
\noindent The maximum and limit values are summarised in Table \ref{t6}. 
    \begin{table}[h]
    \begin{center}
\caption{ Maximum  and limit value of $B$, $s$ and $P$ for different values of $k_d$}\label{t6}
\begin{tabular}{@{}cccccc@{}} \toprule  $k_d$ & $B_{max}$ & $ s_{max} $ & $P_{max}=P^*$ & $s^*$ &  $B^*$ \\  \midrule
    0.03 & 96.6588 & 61.2420 & 32.9487 & 0.1629& 0 \\  \midrule
  0.09 & 57.1921 & 63.5208  & 30.9508 & 1.7971 & 0\\ \midrule
     0.12 & 38.8395 & 65.0240  & 30.3268 &3.0865  & 0 \\ \midrule
     0.18 & 15 & 69.1788  &24.6089  &20.4477  & 0 \\ \botrule
\end{tabular}   \end{center}
\end{table}
\begin{itemize}
            \item[-]the maximum value of the the concentration of the biomass $B$  decreases (see Figs. \ref{kd0.03}, \ref{kd0.09}, \ref{kd0.12}, \ref{kd0.18} and Table \ref{t6} ).
            \item[-]the maximum value of the concentration of the substrate $s$ increases  and so (it is natural) it stabilizes 
             at  a larger value $s^*$ (see Figs. \ref{kd0.03}, \ref{kd0.09}, \ref{kd0.12}, \ref{kd0.18} and  Table \ref{t6} ),
            \item[-] The concentration of $P$
            remains an increasing function but the limit value decreases and is reached earlier (see Figs. \ref{kd0.03}, \ref{kd0.09}, \ref{kd0.12},  \ref{kd0.18} and   Table \ref{t6} ).
        \end{itemize}
\section{Conclusion}
We developed 
an unstructured mathematical model describing the growth kinetics of Trichoderma and the production of enzymes (cellulase) by degradation of a substrate (cellulose). 
This model is more complete than the references cited here. We integrated the hydrolysis step of organic matter in our description. Furthermore, using the theorem of stable and unstable varieties, and Barbalat's lemma, we showed that each trajectory of the system is bounded and converges to one of the non-hyperbolic equilibria depending on the initial conditions. Numerical simulations with data from the literature  
confirm the theoretical study and validate the model.

\noindent Spatialising the model by introducing diffusion would constitute an immediate perspective of this work, which leads to a system of partial differential equations (PDE), of the reaction-diffusion type,  instead of ordinary differential equations. We expect steel to be a global attractor. 
Since Trichoderma positively impacts plant growth, it would be interesting to consider a PDE  system that models this impact and analyzes it mathematically.

\section*{Acknowledgments}
The authors warmly thank Pr. Ouazzani Touhami Amina from {\it Laboratoire des Productions végétales, animals et Agro-industries (Ibn Tofail University)} for all the discussions they had with her, allowing them to enrich their biological knowledge of Trichoderma.

\section*{ORCID}
\noindent asmae hardoul - \url{https://orcid.org/0000-0003-4768-7834}

\noindent zoubida mghazli - \url{https://orcid.org/0000-0003-0264-0637}

\appendix{}
\section{}\label{secA1}
\underline{\bf Determination of equation \eqref{c} }\label{page: EquationZ}\\

Consider the variable  change \eqref{0} and introduce the function $L$ defined by
\begin{align}\label{a}
    L(t):=B(t) z(t)= X(t)+(s(t)-s^*)+\varphi B(t).
\end{align}
Then, using the expressions $\dfrac{dX}{dt}$, $\dfrac{dB}{dt}$ and $\dfrac{ds}{dt}$ given in the system \eqref{model s1}, we have
 \begin{eqnarray*}
   \frac{dL}{dt} &=& -K_H X(t)+\alpha k_d B(t)+K_H X(t)-\frac{1}{Y_{B/s}} \mu(s)B(t)-m_s B(t)+\varphi (\mu(s)B(t)- k_d B(t)).\\
                 &=&(\alpha k_d -m_s)B(t)+
        \left(-\frac{1}{Y_{B/s}} +\varphi\right)\mu(s)B(t)-\varphi k_dB(t).      
 \end{eqnarray*}
 Let 
 \[
\eta_1:=\left(-\frac{1}{Y_{B/s}} +\varphi\right)\mu(s)B(t)\quad\mbox{ and }\quad
\eta_2:=(\alpha k_d -m_s)B(t)-\varphi k_dB(t),
 \]
which can be written as follows:

\begin{eqnarray*}
   \eta_1
          &=&\left[\frac{-1}{Y_{B/s}} +\frac{\frac{\mu(s^*)}{Y_{B/s}}-(\alpha k_d -m_s)}{\mu(s^*)-k_d} \right]\mu(s)B(t) .\\
          &=&\left[\frac{k_d -\alpha Y_{B/s}  k_d+Y_{B/s} m_s}{Y_{B/s}(\mu(s^*)-k_d)}\right]\mu(s)B(t) .\\
          &=&\left[-\frac{k_d (1-\alpha Y_{B/s})+Y_{B/s} m_s}{Y_{B/s}(k_d-\mu(s^*))}\right]\mu(s)B(t) .
 \end{eqnarray*}
 and
 \begin{align*}
     \eta _2 &:=(\alpha k_d-m_s) B(t)-\frac{\frac{\mu(s^*)}{Y_{B/s}}-(\alpha k_d -m_s)}{\mu(s^*)-k_d} k_d B(t).\\
             &=\left[\frac{(\mu(s^*)-k_d)Y_{B/s}(\alpha k_d -m_s) -\mu(s^*)k_d+{Y_{B/s}}(\alpha k_d -m_s)k_d}{Y_{B/s}(\mu(s^*)-k_d)}\right]B(t).\\
             &=\left[ \frac{Y_{B/s}(\alpha k_d -m_s) -k_d}{Y_{B/s}(\mu(s^*)-k_d)}\right]\mu(s^*) B(t).\\
             &=\left[\frac{k_d (1-\alpha Y_{B/s})+Y_{B/s} m_s}{Y_{B/s}(k_d-\mu(s^*))}\right]\mu(s^*) B(t).
 \end{align*}
 then
 \begin{align}\label{b2}
     \frac{dL}{dt}=\eta_1 +\eta_2=-\gamma(\mu(s)-\mu(s^*))B(t).
 \end{align}
 with
 \begin{align*}
     \gamma:=\left[\frac{k_d (1-\alpha Y_{B/s})+Y_{B/s} m_s}{Y_{B/s}(k_d-\mu(s^*))}\right] > 0 .
 \end{align*}
 
 By \eqref{a}, we have
 \begin{align*}
     \frac{dB}{dt} z + B \frac{dz}{dt} =-\gamma(\mu(s)-\mu(s^*))B(t).
 \end{align*}
 Since $B(t)$ is assumed not to vanish, and using the second equation of \eqref{model s1}, we obtain
 \begin{align*}\label{c}
     \frac{dz}{dt} =-\gamma(\mu(s)-\mu(s^*)) -(\mu(s)-k_d)z
 \end{align*}
 which is the equation \eqref{c}.\\

 \underline{\bf Determination of the equation  \eqref{cP}}\label{page:Calcul2}\\

 Consider the variable change \eqref{0P} and let $K$ the function defined by

\begin{equation}
    \label{aP}
    K(t):=B(t) W(t)=(P(t)-P^*)+\omega B(t).
\end{equation} 

Then using 
the equations of the system \eqref{model s1}, we have 
\begin{align*}
   \frac{dK}{dt} &= \left(\frac{1}{Y_{P/s}} \mu(s)+m_P\right) B(t)+\omega \left(\mu(s)- k_d \right)B(t).\\
                 &=\left(\frac{1}{Y_{P/s}} 
                 +\omega
                 \right)\mu(s)B(t)+\left(m_P -\omega k_d
                 \right)B(t).
 \end{align*}
Let
\[
\eta_3:=\left(\frac{1}{Y_{P/s}}      +\omega \right)\mu(s)B(t)\quad\mbox{ and }\quad \eta_4:=\left(m_P -\omega k_d \right)B(t)
\]
  which can be written as follows:
 \begin{align*}
   \eta_3 
          &=\left[\frac{1}{Y_{P/s}} +\frac{\frac{-\mu(s^*)}{Y_{P/s}}-m_P}{\mu(s^*)-k_d} \right]\mu(s)B(t) .\\
          &=\left[\frac{\mu(s^*)-k_d-\mu(s^*)-Y_{P/s} m_P}{Y_{P/s}(\mu(s^*)-k_d)}\right]\mu(s)B(t) .\\
          &=\left[\frac{k_d +Y_{P/s} m_P}{Y_{P/s}(k_d-\mu(s^*))}\right]\mu(s)B(t) .
 \end{align*}
 and
 \begin{align*}
     \eta _4
         &=\left[\frac{(\mu(s^*)-k_d)Y_{P/s}m_P +\mu(s^*)k_d+{Y_{P/s}}m_Pk_d}{Y_{P/s}(\mu(s^*)-k_d)}\right]B(t).\\
             &=\left[ \frac{Y_{P/s}m_P +k_d}{Y_{P/s}(\mu(s^*)-k_d)}\right]\mu(s^*) B(t).\\
             &=\left[-\frac{k_d +Y_{P/s} m_P}{Y_{P/s}(k_d-\mu(s^*))}\right]\mu(s^*) B(t).
 \end{align*}
 then
 \begin{align}\label{b2P}
     \frac{dK}{dt}=\eta_3 +\eta_4=\phi(\mu(s)-\mu(s^*))B(t).
 \end{align}
 with
 \begin{align*}
     \phi:=\left[\frac{k_d +Y_{P/s} m_P}{Y_{P/s}(k_d-\mu(s^*))}\right] > 0 .
 \end{align*}
 According to \eqref{aP}, we have
 \begin{align*}
     \frac{dB}{dt} W + B \frac{dW}{dt} =\phi(\mu(s)-\mu(s^*))B(t).
 \end{align*}
 Using the second equation of \eqref{model s1}, we obtain
 \begin{align*}\label{cP}
     \frac{dW}{dt} =\phi(\mu(s)-\mu(s^*)) -(\mu(s)-k_d)W,
 \end{align*}
which is the equation \eqref{cP}.


\end{document}